\documentclass[11pt]{amsart}

\usepackage{amsmath,latexsym,amssymb,amsthm,amsfonts,graphicx}
\usepackage{subfigure}
\usepackage{xfrac}
\usepackage{fix-cm}
\usepackage{hyperref,doi}
\usepackage{stmaryrd}
\usepackage{multirow}
\usepackage{multicol}
\usepackage{tabularx}
\usepackage{tikz-cd}
\usepackage{enumitem}
\usepackage{wrapfig}
\usepackage[title]{appendix}
\usepackage{color}

\usepackage[margin=1.25in]{geometry}

\setenumerate[1]{label=(\thesection.\arabic*)}

\numberwithin{equation}{section}
\numberwithin{figure}{section}

\newtheorem*{theorem*}{Theorem}
\newtheorem{theorem}{Theorem}[section]
\newtheorem{lemma}[theorem]{Lemma}
\newtheorem{claim}[theorem]{Claim}
\newtheorem{corollary}[theorem]{Corollary}

\theoremstyle{definition}
\newtheorem{remark}[theorem]{Remark}

\newtheorem{examples}[theorem]{Examples}

\def\Z{\ensuremath{\mathbb{Z}}}

\def\R{\ensuremath{\mathbb{R}}}
\def\C{\ensuremath{\mathbb{C}}}
\def\H{\ensuremath{\mathbb{H}}}

\newcommand{\pa}[1]{\left(#1\right)}
\newcommand{\cpa}[1]{\left\{#1\right\}}

\newcommand{\tn}[1]{\textnormal{#1}}
\newcommand{\br}[1]{\left[#1\right]}
\newcommand{\fg}[1]{\left\langle #1\right\rangle}
\newcommand{\im}[1]{\tn{Im} \, #1}
\newcommand{\Int}[1]{\tn{Int} \, #1}

\newcommand{\p}[1]{\pi_1 \pa{#1}}
\newcommand{\card}[1]{\left| #1 \right|}

\newcommand{\cs}{\mathbin{\#}}

\newcommand{\cp}{\C P^2}
\newcommand{\cpb}{\overline{\C P^2}}

\font\cuf=cmtt8
\newcommand{\curl}[1]{{\cuf #1}}

\begin{document}
\title{Artin presentations, triangle groups, and 4-manifolds}

\author[J.~Calcut]{Jack S. Calcut$^\ast$}
\address{Department of Mathematics\\
         Oberlin College\\
         Oberlin, OH 44074}
\email{jcalcut@oberlin.edu}
\urladdr{\href{https://www2.oberlin.edu/faculty/jcalcut/}{\curl{https://www2.oberlin.edu/faculty/jcalcut/}}}

\author[J.~Li]{Jun Li}
\email{junli530@stanford.edu}

\makeatletter
\@namedef{subjclassname@2020}{%
  \textup{2020} Mathematics Subject Classification}
\makeatother

\date{\today}

\begin{abstract}
Gonz{\'a}lez-Acu{\~n}a showed that Artin presentations characterize closed, orientable $3$-manifold groups.
Winkelnkemper later discovered that each Artin presentation determines a smooth, compact, simply-connected $4$-manifold.
We utilize triangle groups to find all Artin presentations on two generators that present the trivial group. We then determine all smooth, closed, simply-connected $4$-manifolds with second betti number at most two that appear in Artin presentation theory.
\end{abstract}

\dedicatory{Dedicated to the memory of Elmar Winkelnkemper.}

\maketitle

\section{Introduction}

An \emph{Artin presentation} is a group presentation $r=\fg{x_1,x_2,\ldots,x_n \mid r_1,r_2,\ldots,r_n}$ such that the following holds in the free group $F_n=\fg{x_1,x_2,\ldots,x_n}$
\[
x_1 x_2 \cdots x_n = (r_1^{-1}x_1 r_1)(r_2 ^{-1} x_2 r_2)\cdots (r_n^{-1} x_n r_n)
\]
Gonz{\'a}lez-Acu{\~n}a~\cite[Thm.~4]{ga} showed that every closed, orientable $3$-manifold admits an open book decomposition with planar page. As a corollary, he obtained the following algebraic characterization of $3$-manifold groups.

\begin{theorem*}[Gonz{\'a}lez-Acu{\~n}a~{{\cite[Thm.~6]{ga}}}]
A group $G$ is the fundamental group of a closed, orientable $3$-manifold if and only if $G$ admits an Artin presentation $r$ for some $n$.
\end{theorem*}

Winkelnkemper~\cite[p.~250]{winkelnkemper} discovered that each Artin presentation $r$ determines not only a closed, orientable $3$-manifold $M^3(r)$ but also a smooth, compact, simply-connected $4$-manifold $W^4(r)$ such that
$\partial W^4(r)=M^3(r)$. All intersection forms are represented by some $W^4(r)$~\cite[pp.~248--250]{winkelnkemper}.
If $M^3(r)$ is the $3$-sphere, then we consider the smooth, \emph{closed}, simply-connected $4$-manifold $X^4(r)=W^4(r) \cup_{\partial} D^4$ obtained from $W^4(r)$ by closing up with a $4$-handle.\\

While all closed, orientable $3$-manifolds appear in Artin presentation theory,
it is unknown which $4$-manifolds appear as a $W^4(r)$ or an $X^4(r)$.
The only contractible manifold $W^4(r)$ is $D^4$ (when $r=\fg{\mid}$ is the empty Artin presentation).
So, no Mazur manifold appears as a $W^4(r)$.
Nevertheless, there are no known smooth, closed, simply-connected $4$-manifolds that do not appear as an $X^4(r)$;
many interesting closed $4$-manifolds are known to appear this way including all elliptic surfaces $E(n)$ where $E(2)$ is diffeomorphic to the Kummer surface $K3$~\cite{cw}.\\

We determine all closed $4$-manifolds $X^4(r)$ where $r$ is an Artin presentation on two generators.
(For $n=0$ and $n=1$, the problem is straightforward: only $S^4$, $\cp$, and $\cpb$ appear.)
Theorem~\ref{fms} gives the complete list of these manifolds:
$\cp\cs\cp$, $\cp\cs\cpb$, $\cpb\cs\cpb$, and $S^2\times S^2$.
Exotic simply-connected, closed $4$-manifolds are currently not known to exist with second betti number $\leq2$.
Theorem~\ref{fms} shows that such manifolds, whether or not they exist in general, do not appear in
Artin presentation theory.
Exotic $4$-manifolds do appear in Artin presentation theory with second betti number $\geq 10$~\cite{calcutjuggle}.
We conjecture that closed, exotic $4$-manifolds appear in Artin presentation theory with second betti number three,
and that this relates to the Torelli subgroup in Artin presentation theory (see~\cite[p.~250]{winkelnkemper}
and~\cite{calcutjuggle}).\\

Our proof of Theorem~\ref{fms} uses the classification of Artin presentations on two generators and
properties of classical triangle groups to find all Artin presentations on two generators that
present the trivial group. We then introduce a move on Artin presentations on two generators
which preserves the $4$-manifolds. Using this move and the Kirby calculus, we then identify the $4$-manifolds.
It would be interesting to find other such moves in Artin presentation theory.
Armas-Sanabria has shown certain three generator Artin presentations present nontrivial groups~\cite{lorena}.\\

For each nonnegative integer $n$, let $\mathcal{R}_n$ denote the set of Artin presentations on $n$ generators.
We include a proof of the folklore theorem (see~\cite[p.~245]{winkelnkemper} and~\cite[p.~10]{ga})
that $\mathcal{R}_n$ is a group isomorphic to the product $P_n\times\Z^n$ of the
pure braid group $P_n$ with the rank $n$ free abelian group.\\

Throughout, $X \approx Y$ means that $X$ is orientation preserving diffeomorphic to $Y$.
If $X$ is an oriented manifold, then $\overline{X}$ denotes the same manifold with the opposite orientation.

\section*{Acknowledgement}

The authors thank both referees for their helpful comments.

\section{Artin Presentations}

In this section, we review fundamental properties of Artin presentations and fix notation.
We begin by recalling how each Artin presentation arises naturally
from a homeomorphism of a compact $2$-disk with holes.\\

Let $\Omega_n$ denote the compact $2$-disk with $n$ holes as in Figure~\ref{fig:omegan}.
\begin{figure}[htbp!]
    \centerline{\includegraphics[scale=1.0]{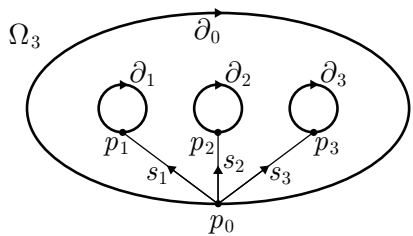}}
    \caption{Compact $2$-disk with $n$ holes denoted $\Omega_n$ (the case $n=3$ is depicted).}
\label{fig:omegan}
\end{figure}
The boundary components $\partial_0,\partial_1,\ldots,\partial_n$ of $\Omega_n$ are parameterized clockwise
and are based at $p_0,p_1,\ldots,p_n$ respectively.
For each $1\leq i \leq n$, let $s_i$ be an oriented segment from $p_0$ to $p_i$ as in Figure~\ref{fig:omegan}.
Given a path $\alpha$, let $\overline{\alpha}$ denote the reverse path
and let $\br{\alpha}$ denote the path homotopy class of $\alpha$.
Concatenation of paths---performed left to right---and the induced operation on classes will be denoted by juxtaposition.
For each $1\leq i \leq n$, let $x_i=\br{s_i \partial_i \overline{s_i}}$ as in Figure~\ref{fig:omegangen}.
\begin{figure}[htbp!]
    \centerline{\includegraphics[scale=1.0]{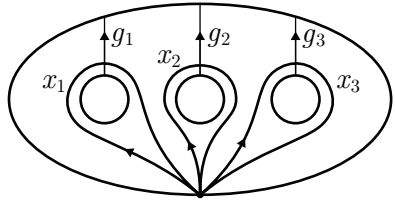}}
    \caption{Generators $x_i$ of $\p{\Omega_3,p_0}$.}
\label{fig:omegangen}
\end{figure}
So, $\p{\Omega_n,p_0}\cong F_n=\fg{x_1,x_2,\ldots,x_n}$ is free of rank $n$.\\

Let $h:\Omega_n \to \Omega_n$ be a homeomorphism that equals the identity (point-wise)
on the boundary of $\Omega_n$.
Then, the induced homomorphism $h_{\sharp}: \p{\Omega_n,p_0} \to \p{\Omega_n,p_0}$ is an automorphism.
For each $1\leq i \leq n$, define $r_i=\br{s_i (h\circ\overline{s_i})} \in F_n$.
Define the presentation $r=r(h)=\fg{x_1,x_2,\ldots,x_n \mid r_1,r_2,\ldots,r_n}$.

\begin{claim}\label{hsharpxi}
For each $1\leq i \leq n$, we have $h_{\sharp}(x_i)=r_i^{-1} x_i r_i$.
\end{claim}

\begin{proof}
Note that $h\circ\overline{s_i}=\overline{h\circ s_i}$, $r_i^{-1}=\br{(h\circ s_i)\overline{s_i}}$,
and $h\circ\partial_i = \partial_i$.
Therefore
\begin{align*}
h_{\sharp}(x_i) &= \br{h\circ (s_i\partial_i \overline{s_i})} =
\br{h\circ s_i}\br{h\circ\partial_i}\br{h\circ\overline{s_i}}\\
								&= \br{h\circ s_i}\br{\overline{s_i}}\br{s_i}\br{\partial_i}
								\br{\overline{s_i}}\br{s_i}\br{h\circ\overline{s_i}}\\
								&=r_i^{-1} x_i r_i
\end{align*}
\end{proof}

\begin{claim}\label{hgivesartin}
The presentation $r=r(h)$ determined by $h$ is an Artin presentation.
\end{claim}

\begin{proof}
The following holds in $F_n$
\begin{align*}
x_1 x_2 \cdots x_n &= \br{\partial_0} = \br{h\circ\partial_0}\\
									 &= h_{\sharp}(x_1 x_2 \cdots x_n) = h_{\sharp}(x_1)h_{\sharp}(x_2)\cdots h_{\sharp}(x_n)\\
								   &=(r_1^{-1}x_1 r_1)(r_2 ^{-1} x_2 r_2)\cdots (r_n^{-1} x_n r_n)
\end{align*}
where the last equality used Claim~\ref{hsharpxi}.
\end{proof}

Let $\tn{Homeo}\pa{\Omega_n,\partial\Omega_n}$ denote the group of homeomorphisms
of $\Omega_n$ that equal the identity (point-wise) on $\partial\Omega_n$.
By Claim~\ref{hgivesartin}, we have a function
\[
\psi:\tn{Homeo}\pa{\Omega_n,\partial\Omega_n} \to \mathcal{R}_n
\]
given by $\psi\pa{h}=r(h)$.
If $h$ and $h'$ are isotopic relative to $\partial\Omega_n$, then $\psi(h)=\psi(h')$.
We will show that $\psi$ is a surjective homomorphism of groups.
First, we present a few examples.
Given an Artin presentation $r\in\mathcal{R}_n$, we define $\pi(r)$ to be the group presented by $r$
and $A(r)$ to be the exponent sum matrix of $r$ meaning
$\br{A(r)}_{ij}$ equals the exponent sum of $x_i$ in $r_j$.
Note that $A(r)$ is an $n\times n$ integer matrix,
the abelianization of $\pi(r)$ is isomorphic to $\Z^n/\tn{Im}A$ where $\tn{Im}A$
denotes the image of $A:\Z^n\to\Z^n$,
and $\pi(r)$ is perfect if and only if
$A(r)$ is unimodular (that is, $\det A=\pm1$).\\

\begin{examples}\label{APex}
\noindent
\begin{enumerate}[label=(\arabic*),leftmargin=*]\setcounter{enumi}{0}
\item The empty presentation $\varepsilon=\fg{\mid}$ is the unique Artin presentation in $\mathcal{R}_0$.
Here, $\pi(\varepsilon)$ is trivial and $A(r)=\br{}$ is empty.
\item The Artin presentations associated to $T_c$ and $T_c^{-1}$ in Figure~\ref{fig:dehn1}
are $\fg{x_1 \mid x_1}$ and $\fg{x_1 \mid x_1^{-1}}$ respectively.
\begin{figure}[htbp!]
    \centerline{\includegraphics[scale=1.0]{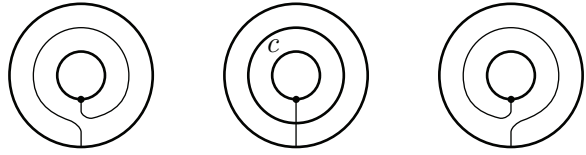}}
    \caption{Simple closed curve $c$ in $\Omega_1$ (center), result of right Dehn twist $T_c$ about $c$ (right), and result of left Dehn twist $T_c^{-1}$ about $c$ (left).}
\label{fig:dehn1}
\end{figure}
Given a homeomorphism $h\in \tn{Homeo}\pa{\Omega_n,\partial\Omega_n}$, one computes the associated
Artin presentation $r=\psi(h)$ as follows.
Recall Figures~\ref{fig:omegan} and~\ref{fig:omegangen}.
The relation $r_i$ is obtained by starting at $p_i$, following
$h(\overline{s_i})$, and recording $x_j$ (respectively $x_j^{-1}$) each time $g_j$ is crossed from left to right (respectively right to left).
\item Each $r\in\mathcal{R}_1$ has the form $r=\fg{x_1 \mid x_1^a}$ for some integer $a$.
Here, $\pi(r)\cong\Z/\card{a}\Z$ and $A(r)=\br{a}$.
\item Let $a_1,a_2,\ldots,a_n$ be any integers.
Then, $r=\fg{x_1,x_2,\ldots,x_n \mid x_1^{a_1},x_2^{a_2},\ldots,x_n^{a_n}}$ is an Artin presentation in $\mathcal{R}_n$ that presents the free product of cyclic groups $\Z/\card{a_i}\Z$.
\item The Artin presentations associated to $T_c$ and $T_c^{-1}$ in Figure~\ref{fig:dehn2}
are $\fg{x_1,x_2 \mid x_1 x_2, x_1 x_2}$ and $\fg{x_1,x_2 \mid x_2^{-1} x_1^{-1}, x_2^{-1} x_1^{-1}}$ respectively.
\begin{figure}[htbp!]
    \centerline{\includegraphics[scale=1.0]{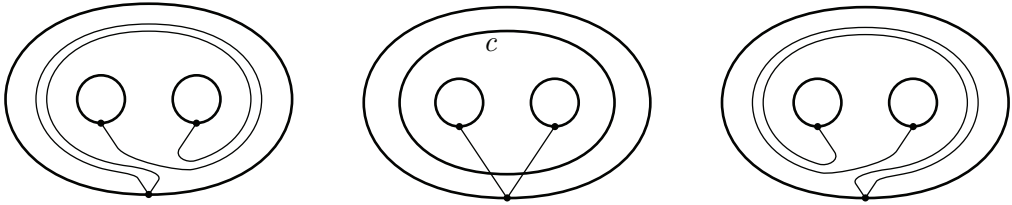}}
    \caption{Simple closed curve $c$ in $\Omega_2$ (center), result of right Dehn twist $T_c$ about $c$ (right), and result of left Dehn twist $T_c^{-1}$ about $c$ (left).}
\label{fig:dehn2}
\end{figure}
\item\label{charR2} Each $r=\fg{x_1,x_2 \mid r_1,r_2}\in\mathcal{R}_2$ has the form $r_1=x_1^{a-c}(x_1 x_2)^c$ and
$r_2=x_2^{b-c}(x_1 x_2)^c$ for some integers $a$, $b$, and $c$
(see~\cite[p.~360]{calcutalgebraic} and~\cite[p.~245]{winkelnkemper}).
We denote such an Artin presentation by $r(a,b,c)$, so
$A(r(a,b,c))=
\begin{bmatrix}
a & c\\
c & b
\end{bmatrix}$.
For instance, $r(-1,-3,2)$ presents the binary icosahedral group $I(120)$.
We mention that $\pi(r(a,b,c))\cong\pi(r(-a,-b,-c))$ by the map $x_i\mapsto x_i$.
\end{enumerate}
\end{examples}

\begin{remark}\label{operation}
For each $n$, $\mathcal{R}_n$ has a natural binary operation.
Namely, let $r,u\in\mathcal{R}_n$.
Define the composition $u\circ r$ to be the presentation
$t=\fg{x_1,x_2,\ldots,x_n \mid t_1,t_2,\ldots,t_n}$ as follows:
let $R_i$ be obtained by substituting $u_j^{-1} x_j u_j$ for $x_j$ in $r_i$,
and define $t_i=u_i R_i$.
We now show that $t\in\mathcal{R}_n$ \emph{provided} $r,u\in\im{\psi}$
(shortly, we will see that $\im{\psi}=\mathcal{R}_n$, which means this operation is defined on all of $\mathcal{R}_n$).
Let $h,k\in \tn{Homeo}\pa{\Omega_n,\partial\Omega_n}$ such that $\psi(h)=r$ and $\psi(k)=u$.
For each $1\leq j \leq n$, Claim~\ref{hsharpxi} implies that
$h_{\sharp}(x_j)=r_j^{-1} x_j r_j$ and $k_{\sharp}(x_j)=u_j^{-1} x_j u_j$.
So, for each $1\leq i \leq n$, we have
\begin{align*}
\pa{k\circ h}_{\sharp}(x_i) &= k_{\sharp}\pa{h_{\sharp}(x_i)}=k_{\sharp}\pa{r_i^{-1} x_i r_i}\\
									 &=k_{\sharp}(r_i)^{-1} u_i^{-1} x_i u_i k_{\sharp}(r_i) \\
								   &=\pa{u_i k_{\sharp}(r_i)}^{-1} x_i \pa{u_i k_{\sharp}(r_i)}\\
									&=\pa{u_i R_i}^{-1} x_i \pa{u_i R_i}\\
									&=t_i^{-1} x_i t_i
\end{align*}
Thus
\[
x_1 x_2 \cdots x_n 	= \pa{k\circ h}_{\sharp}(x_1 x_2 \cdots x_n)
										=(t_1^{-1}x_1 t_1)(t_2 ^{-1} x_2 t_2)\cdots (t_n^{-1} x_n t_n)
\]
and so $t=u\circ r \in\mathcal{R}_n$.
Summarizing, $\left.\psi\right|:\tn{Homeo}\pa{\Omega_n,\partial\Omega_n} \to \im{\psi}$ is surjective,
the domain is a group, the codomain is a set with a binary operation,
and $\psi$ respects the operations---meaning $\psi(k\circ h)=\psi(k)\circ\psi(h)$.
These facts imply that $\im{\psi}$ is a group and $\left.\psi\right|$ is a surjective homomorphism.
Note that each Artin presentation in the examples above lies in $\im{\psi}$.
The identity in $\im{\psi}$ is $\fg{x_1,x_2,\ldots,x_n \mid 1,1,\ldots,1}$.
\end{remark}

Let $D_n$ denote the compact $2$-disk with $n\geq0$ marked points $q_1,q_2,\ldots,q_n$
in its interior as in Figure~\ref{fig:D_n}.
\begin{figure}[htbp!]
    \centerline{\includegraphics[scale=1.0]{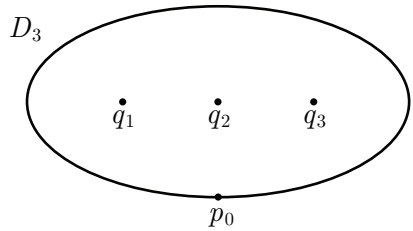}}
    \caption{Surface $D_n$ that is the compact $2$-disk $D^2$ with $n$ marked points $q_1,q_2,\ldots,q_n$
		in its interior (the case $n=3$ is depicted).}
\label{fig:D_n}
\end{figure}
We choose $q_i$ to lie at the center of the simple
closed curve $\partial_i$ and let $Q=\cpa{q_1,q_2,\ldots,q_n}$.
Let $D^2-Q$ be the $2$-disk with $n\geq0$ punctures.\\

\begin{claim}\label{psisurjective}
The function $\psi$ is surjective. Hence,
$\psi:\tn{Homeo}\pa{\Omega_n,\partial\Omega_n} \to \mathcal{R}_n$ is a surjective homomorphism of groups.
\end{claim}

\begin{proof}
Let $t\in\mathcal{R}_n$.
It suffices to prove that $t\in\im{\psi}$ since then the second conclusion follows from Remark~\ref{operation}.
Define the endomorphism $\beta:F_n\to F_n$ by $x_i\mapsto t_i^{-1} x_i t_i$.
By Artin\footnote{Artin's beautiful algebraic argument appeared originally in German
and later in English~\cite[pp.~114--115]{artin47};
Birman gave an exposition of Artin's argument in her book~\cite[Thm.~1.9, p.~30]{birman}.}~\cite[pp.~64--68]{artin},
the map $\beta$ is a pure braid group automorphism
of $F_n$.
So, there exists a homeomorphism $h'$ of $D^2-Q$ that is a product of homeomorphisms
corresponding to braid group generators such that: $h'$ is the identity on $\partial D^2$,
$h'$ sends each puncture to itself (by purity of $\beta$), and $h'_{\sharp}=\beta$.
Further, we can and do assume that the restriction $h$ of $h'$ to $\Omega_n$
equals the identity on $\partial\Omega_n$.
Hence, $h\in \tn{Homeo}\pa{\Omega_n,\partial\Omega_n}$ and $h_{\sharp}=\beta$.
Let $r=\psi(h)\in\mathcal{R}_n$.
For each $1\leq i\leq n$, $h_{\sharp}(x_i)=r_i^{-1} x_i r_i$ by Claim~\ref{hsharpxi}.
So, $t_i^{-1} x_i t_i=r_i^{-1} x_i r_i$ and $x_i\pa{t_i r_i^{-1}}=\pa{t_i r_i^{-1}}x_i$.
Commuting elements in $F_n$ are powers of the same word~\cite[p.~42]{mks}.
Thus, $t_i r_i^{-1}=x_i^{a_i}$ for some integer $a_i$, which implies $t_i=x_i^{a_i}r_i$.
Define $u=\fg{x_1,x_2,\ldots,x_n \mid x_1^{a_1},x_2^{a_2},\ldots,x_n^{a_n}}\in\mathcal{R}_n$.
By the examples above, there exists $k\in \tn{Homeo}\pa{\Omega_n,\partial\Omega_n}$
such that $\psi(k)=u$.
As $u_j^{-1}x_j u_j=x_j$ for each $1\leq j \leq n$, we have $t=u\circ r \in\im{\psi}$ as desired.
\end{proof}

\begin{remark}\label{Ahom}
Let $\tn{Mat}(n,\Z)$ be the additive group of $n\times n$ integer matrices.
The function $A:\mathcal{R}_n \to \tn{Mat}(n,\Z)$ is a homomorphism---meaning $A(u\circ r)=A(u)+A(r)$.
To see this, note that $u\circ r\in\mathcal{R}_n$ by Claim~\ref{psisurjective}.
By definition, $\br{A(u\circ r)}_{ij}$ equals the exponent sum of $x_i$ in $u_j R_j$
where $R_j$ is obtained from $r_j$ by substituting $u_k^{-1} x_k u_k$ for $x_k$.
Therefore, $\br{A(u\circ r)}_{ij}=\br{A(u)}_{ij}+\br{A(r)}_{ij}$ as desired.
\end{remark}

Let $\tn{Homeo}_0\pa{\Omega_n,\partial\Omega_n}$ be the subgroup of $\tn{Homeo}\pa{\Omega_n,\partial\Omega_n}$
consisting of homeomorphisms (fixed point-wise on $\partial\Omega_n$) that are
isotopic relative to $\partial\Omega_n$ to the identity.

\begin{claim}\label{kerpsi}
The kernel of $\psi$ equals $\tn{Homeo}_0\pa{\Omega_n,\partial\Omega_n}$.
\end{claim}

Before we prove Claim~\ref{kerpsi}, we need a technical lemma.
Incidentally, this lemma is the reason we may---and typically do---assume the
individual relations $r_i$ in each Artin presentation $r$ are freely reduced
(see also~\cite[pp.~363--365]{calcutalgebraic}).

\begin{lemma}\label{diskinomegan}
Let $h$ be a diffeomorphism of $\Omega_n$ fixed point-wise on $\partial \Omega_n$.
Assume that the arcs $h(\overline{s_i})$ meet the segments $g_j$ (see Figure~\ref{fig:omegangen}) in general position.
Suppose some $h(\overline{s_i})$ crosses some $g_j$ ($j=i$ is possible) at the point $a$ and then,
without crossing any other segment $g_k$,
$h(\overline{s_i})$ crosses $g_j$ in the opposite direction at the point $b$ as in Figure~\ref{fig:lemmafig}.
\begin{figure}[htbp!]
    \centerline{\includegraphics[scale=1.0]{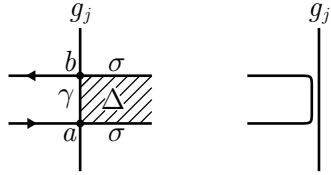}}
    \caption{Arc $h(\overline{s_i})$ meeting $g_j$ consecutively in opposite directions
		without meeting any segments $g_k$ in-between (left) and result of isotopy pushing $\sigma$ across $\Delta$ to a parallael copy of $\gamma$ (right).}
\label{fig:lemmafig}
\end{figure}
Then, there is an isotopy of $h$ relative to $\partial\Omega_n$ 
that eliminates the crossings $a$ and $b$.
This isotopy introduces no new crossings between the arcs $h(\overline{s_k})$ and $g_l$,
though it may eliminate other crossings (in case some $h(\overline{s_k})$ meets $\gamma$).
\end{lemma}

\begin{proof}[Proof of Lemma~\ref{diskinomegan}]
Let $\gamma$ and $\sigma$ be the arcs between $a$ and $b$ in $g_j$ and $h(\overline{s_i})$ respectively.
So, $C=\gamma\cup\sigma$ is a simple closed curve in $\Int{\Omega_n}$.
By the Schoenflies theorem, $C$ bounds a $2$-disk $\Delta\subseteq \Int{D^2}$.
We show $\Delta$ is disjoint from $\partial\Omega_n$.
As $\Delta\subseteq\Int{D^2}$, $\Delta$ is disjoint from $\partial_0$.
Suppose, by way of contradiction, that $\Delta$ meets $\partial_k$ for some $1\leq k\leq n$.
Then, $\partial_k \subseteq \Int{\Delta}$ by connectedness.
If $k\neq j$, then the segment $g_k$ must meet $C$ by the Jordan Curve Theorem JCT.
Hence, $g_k$ meets $\sigma$ which contradicts the hypotheses on $h(\overline{s_i})$.
Next, assume $k=j$. 
Observe that the segment $g_j$ begins on $\partial_j$ in $\Int{\Delta}$.
Also, an open segment of $g_j$ ending at the lower point $a$ or $b$
($a$ in Figure~\ref{fig:lemmafig}) does not lie in $\Delta$.
Therefore, $g_j$ must meet $\sigma$ at a point other than $a$ or $b$ by the JCT.
Again, this contradicts the hypotheses on $h(\overline{s_i})$.
Therefore, $\Delta\cap\partial\Omega_n=\emptyset$ and $\Delta\subseteq \Int{\Omega_n}$.
The isotopy is now obtained by pushing $\sigma$ across $\Delta$
to a parallel copy of $\gamma$ as in Figure~\ref{fig:lemmafig}.
\end{proof}

\begin{proof}[Proof of Claim~\ref{kerpsi}]
Let $h\in\ker{\psi}$, and let $r=\psi(h)$.
By an isotopy of $h$ relative to $\partial \Omega_n$, we can and do assume $h$ is a diffeomorphism and
the arcs $h(\overline{s_i})$ meet the segments $g_j$ in general position.
As $h\in\ker{\psi}$, each $r_i=1$ which means $r_i$ freely reduces to $1$.
Let $x_j^{\pm1} x_j^{\mp1}$ be adjacent letters in some $r_i$ ($j=i$ is possible);
this corresponds to $h(\overline{s_i})$ crossing some $g_j$ and then,
without crossing any other segment $g_k$,
crossing $g_j$ in the opposite direction.
By Lemma~\ref{diskinomegan}, an isotopy of $h$ relative to $\partial\Omega_n$ eliminates these crossings.
Repeating this operation finitely many times, we may assume that the arcs
$h(\overline{s_i})$ are disjoint from the segments $g_j$.
By a small isotopy of $h$ relative to $\partial \Omega_n$,
we may assume height restricts to a Morse function on the arcs $h(s_i)$ with distinct critical values.
The minimal local maximum of height on $h(s_1)$ may be ambiently cancelled 
with an adjacent local minimum (see~\cite[pp.~1845--1852]{cks}, especially Figure~11).
Repeating this operation finitely many times, the arc $h(s_1)$ has strictly increasing height.
By integrating an appropriate horizontal vector field on $\Omega_n$, we may assume $h(s_1)=s_1$.
Integrating a vector field tangent to $s_1$, we may assume $h$ also equals the identity on $s_1$.
One may repeat this procedure on $s_2$---without disturbing $s_1$---and so forth.
Thus, we have $h$ is a diffeomorphism of $\Omega_n$ equal to the identity on $\partial \Omega_n$ and on the 
segments $s_i$ for $1\leq i \leq n$.
Cutting $\Omega_n$ open along the arcs $s_i$, $h$ becomes a homeomorphism of the $2$-disk equal to the identity on boundary. By Alexander's trick~\cite[47--48]{fm}, this homeomorphism is isotopic to the identity relative to boundary.
This latter isotopy induces an isotopy relative to $\partial\Omega_n$ of $h$ to the identity.
\end{proof}

The mapping class group of $\Omega_n$ is
\[
\tn{Mod}\pa{\Omega_n} = \tn{Homeo}\pa{\Omega_n,\partial\Omega_n} / \tn{Homeo}_0\pa{\Omega_n,\partial\Omega_n}
\]
For useful equivalent definitions of $\tn{Mod}\pa{\Omega_n}$, see~\cite[pp.~44--45]{fm}. 
Claims~\ref{psisurjective} and~\ref{kerpsi} immediately imply the following.

\begin{corollary}\label{ModRn}
The function $\tn{Mod}\pa{\Omega_n} \to \mathcal{R}_n$ given by $\br{h}\mapsto \psi(h)$ is an isomorphism.
\end{corollary}

Recall from Figure~\ref{fig:D_n} above that $D_n$ is the compact $2$-disk with $n$ marked points in its interior.
The mapping class group of $D_n$ is
\[
\tn{Mod}\pa{D_n}=\tn{Homeo}\pa{D_n,\partial D_n} / \tn{Homeo}_0\pa{D_n,\partial D_n}
\]
where homeomorphisms and isotopies fix $\partial D_n = S^1$ point-wise and
the marked points may be permuted~\cite[p.~45]{fm}.
The pure mapping class group $\tn{PMod}\pa{D_n}$ is the subgroup of $\tn{Mod}\pa{D_n}$ that fixes each marked point individually~\cite[pp.~90]{fm}.
Let $B_n$ denote the classical $n$ strand braid group and $P_n$ denote the $n$ strand pure braid group.
There are canonical isomorphisms $B_n \cong \tn{Mod}\pa{D_n}$ and
$P_n \cong \tn{PMod}\pa{D_n}$~\cite[pp.~243--249]{fm}.
Recall that $D^2-Q$ is the $2$-disk with $n$ punctures.
One may regard marked points as punctures (see~\cite[p.~45]{fm}) in the sense that
there are canonical isomorphisms $\tn{Mod}\pa{D_n} \cong \tn{Mod}\pa{D^2-Q}$ and
$\tn{PMod}\pa{D_n} \cong \tn{PMod}\pa{D^2-Q}$.
\\

\begin{claim}\label{modPZ}
There is a canonical isomorphism $\tn{Mod}\pa{\Omega_n} \cong P_n \times \Z^n$.
\end{claim}

\begin{proof}
We will define the homomorphisms in the sequence
\begin{equation}\label{eq:ses}
\begin{tikzcd}
1 \arrow[r] & \Z^n \arrow[r,"\delta"] & \tn{Mod}\pa{\Omega_n} \arrow[r,"\left.\eta\right|"]
	\arrow[l,bend left,dashed,"\mu"] & \tn{PMod}\pa{D_n} \arrow[r] & 1
\end{tikzcd}
\end{equation}
and show the sequence is exact and left split.
For each $1\leq i \leq n$, let $T_i$ denote a right Dehn twist about a simple closed curve in $\Omega_n$
parallel to $\partial_i$ as in Figure~\ref{fig:dehn1}.
If $a=(a_1,a_2,\ldots,a_n)\in\Z^n$, then define $\delta(a)=\br{T_1^{a_1} T_2^{a_2} \cdots T_n^{a_n}}$.
As these Dehn twists commute, $\delta$ is a homomorphism.
The map $\mu$ is defined to be the composition: first the isomorphism in Corollary~\ref{ModRn},
second the homomorphism $A$ (see Remark~\ref{Ahom} above), and third return the diagonal.
As each of these three maps is a homomorphism, $\mu$ is a homomorphism.
Observe that $\mu\circ\delta=\tn{id}$ on $\Z^n$.
So, $\delta$ is injective and~\eqref{eq:ses} is exact on the left.\\

There is a canonical homomorphism $\eta:\tn{Mod}\pa{\Omega_n} \to \tn{Mod}\pa{D_n}$ induced by
the inclusion $\Omega_n \subseteq D_n$~\cite[pp.~82--84]{fm}.
By~\cite[Thm.~3.18, p.~84]{fm} (or~\cite[Prop.~3.19, p.~85]{fm} for $n=1$),
$\ker{\eta}=\fg{\br{T_1},\br{T_2},\ldots,\br{T_n}}_{\tn{FA}}\cong \Z^{n}$
where $\tn{FA}$ indicates free abelian group.
By the definition of $\tn{Mod}\pa{\Omega_n}$, $\im{\eta}\subseteq \tn{PMod}\pa{D_n}$
and we have the restriction homomorphism $\left.\eta\right|:\tn{Mod}\pa{\Omega_n} \to \tn{PMod}\pa{D_n}$.
As $\ker{\eta}=\ker{\left.\eta\right|}$, the sequence~\eqref{eq:ses} is exact in the middle.
Let $\br{h}\in \tn{PMod}\pa{D_n}$. By an isotopy relative to $\partial D_n$ and $Q$,
we may assume $h$ also equals the identity on and inside each $\partial_i$ for $1\leq i \leq n$.
If $h'$ is the restriction of $h$ to $\Omega_n$,
then $\left.\eta\right|(\br{h'})=\br{h}$.
So, $\left.\eta\right|$ is surjective and~\eqref{eq:ses} is exact on the right.\\

Therefore,~\eqref{eq:ses} is short exact and left split.
By~\cite[Prop.~26, p.~385]{dummitfoote}, there is an induced isomorphism
$\tn{Mod}\pa{\Omega_n} \xrightarrow{\cong} \tn{PMod}\pa{D_n} \times \Z^n \cong P_n\times \Z^n$.
\end{proof}

Corollary~\ref{ModRn} and Claim~\ref{modPZ} prove the following folklore theorem.

\begin{theorem}\label{folklore}
For each $n\geq0$, there are canonical isomorphisms
\[
\mathcal{R}_n \cong \tn{Mod}\pa{\Omega_n} \cong  P_n\times\Z^n
\] 
\end{theorem}

\begin{examples}
Figure~\ref{fig:braids} depicts five framed pure braids corresponding (from left to right) to:
the Dehn twists $T_c^{-1}$ and $T_c$ from Figure~\ref{fig:dehn1},
the Dehn twists $T_c^{-1}$ and $T_c$ from Figure~\ref{fig:dehn2},
and the Artin presentation $r(a,b,c)$.
\begin{figure}[htbp!]
    \centerline{\includegraphics[scale=1.0]{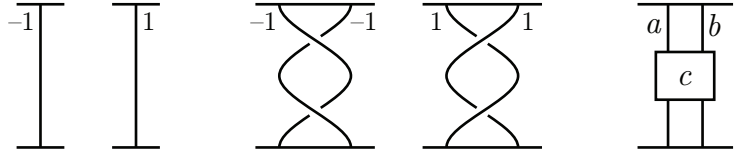}}
    \caption{Framed pure braids.}
\label{fig:braids}
\end{figure}
\end{examples}

We now shift our attention to $3$- and $4$-manifolds in Artin presentation theory.
Each Artin presentation $r\in\mathcal{R}_n$ determines the following.
\begin{align*}
\pi(r)	& \quad \tn{the group presented by $r$}\\
A(r)		& \quad \tn{the exponent sum matrix of $r$}\\
h(r)		&	\quad \tn{a self-diffeomorphism of $\Omega_n$}\\
M^3(r)	&	\quad \tn{a closed, oriented $3$-manifold}\\
W^4(r)	& \quad \tn{a smooth, compact, simply-connected, oriented $4$-manifold}
\end{align*}
By Corollary~\ref{ModRn}, $r$ determines a self-diffeomorphism $h=h(r)$ of $\Omega_n$ equal to the identity
on $\partial \Omega_n$ and unique up to isotopy relative to $\partial\Omega_n$.
The $3$-manifold $M^3(r)$ is defined by Winkelnkemper's open book construction with planar page $\Omega_n$
(see Gonz{\'a}lez-Acu{\~n}a~\cite{ga} and Winkelnkemper~\cite{winkelnkemper}).
Namely, consider the mapping torus $\Omega(h)$ of $h$ which is obtained from 
$\Omega_n\times[0,1]$ by identifying $(x,1)$ with $(h(x),0)$ for each $x\in\Omega_n$.
The boundary of $\Omega(h)$ equals $\pa{\partial\Omega_n}\times S^1$,
and $M^3(r)$ is obtained from $\Omega(h)$ by gluing on $\pa{\partial\Omega_n}\times D^2$
using the identity function on $\pa{\partial\Omega_n}\times S^1$.
The fundamental group of $M^3(r)$ is isomorphic to $\pi(r)$ (see~\cite[p.~10]{ga} or~\cite[p.~247]{winkelnkemper}).
In particular, $M^3(r)$ is an integer homology $3$-sphere if and only if $A(r)$ is unimodular.
Using the symplectic property of closed surface homeomorphisms, Winkelnkemper observed that $A(r)$ is always
symmetric for an Artin presentation $r$ (see~\cite[p.~250]{winkelnkemper}, or see~\cite{calcutalgebraic}
for an algebraic proof of this fact).
This led Winkelnkemper to discover that $r$ determines a $4$-manifold
using a sort of relative open book construction as follows.
Embed $\Omega_n$ in $S^2$, and let $C$ be the closure in $S^2$ of the complement of $\Omega_n$
(so, $C$ is a disjoint union of $n+1$ smooth $2$-disks).
Extend $h$ to $S^2$ then to $D^3$, and let $H$ be the resulting self-diffeomorphism 
of $D^3$. The mapping torus $W(H)$ of $H$ contains $C\times S^1$ in its boundary.
Then, $W^4(r)$ is obtained from $W(H)$ by gluing on $C\times D^2$ in the canonical way.
In particular, $\partial W^4(r)=M^3(r)$ and the intersection form of $W^4(r)$ is given by $A(r)$.
If $M^3(r)$ is the $3$-sphere, then we define $X^4(r)=W^4(r) \cup_{\partial} D^4$ a smooth, closed, simply-connected, oriented $4$-manifold (that is, we close up with a $4$-handle). By Cerf's theorem~\cite{cerf}, a $4$-handle may be added in an essentially unique way, and so $X^4(r)$ is well-defined.\\

An alternative definition of $W^4(r)$ is as follows (see also~\cite[$\S$2]{cw}).
Let $r\in\mathcal{R}_n$ be an Artin presentation. 
By Theorem~\ref{folklore}, $r$ determines an integer framed pure braid.
The framing of the $i$th strand is $\br{A(r)}_{ii}$.
Let $L(r)$ be the framed pure link in $S^3=\partial D^4$ obtained as the closure of this framed pure braid.
Define $W^4(r)$ to be $D^4$ union $n$ $2$-handles attached along $L(r)$.
So, $L(r)$ is a Kirby diagram for $W^4(r)$ (see~\cite[p.~115]{gs} for an introduction to Kirby diagrams).
Figure~\ref{fig:R2diagram} gives Kirby diagrams for $W^4(r)$ where $r\in\mathcal{R}_1$ and
$r\in\mathcal{R}_2$.
\begin{figure}[htbp!]
    \centerline{\includegraphics[scale=1.0]{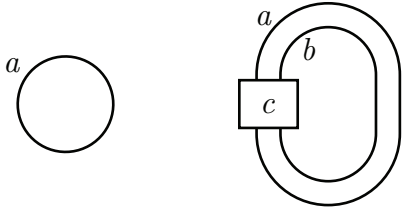}}
    \caption{Kirby diagrams for $W^4(r)$ where $r=\fg{x_1 \mid x_1^a}\in\mathcal{R}_1$ (left) and $r=r(a,b,c)\in\mathcal{R}_2$ (right).}
\label{fig:R2diagram}
\end{figure}
For example, $X^4(\fg{x_1\mid x_1})\approx\cp$ and $X^4(\fg{x_1\mid x_1^{-1}})\approx\cpb$.\\

\begin{remark}\label{basicdiffeos}
Three basic diffeomorphisms between the $4$-manifolds $W^4(r(a,b,c))$ are as follows
where $a$, $b$, and $c$ are any integers.
\begin{align}
W^4(r(a,b,c)) &\approx W^4(r(b,a,c)) \label{symmetry}\\
W^4(r(a,b,1)) &\approx W^4(r(a,b,-1)) \label{hopf}\\
W^4(r(-a,-b,-c)) &\approx \overline{W^4(r(a,b,c))} \label{inverse}
\end{align}
The first two diffeomorphisms are given by simple isotopies of the Kirby diagrams:
for~\eqref{symmetry} interchange the two link components, and for~\eqref{hopf} flip one component to switch the sign of the single twist.
Regarding~\eqref{hopf}, we mention that while the links are not equivalent as \emph{oriented} links,
the resulting $4$-manifolds are independent of the orientations of the link components.
The third diffeomorphism is a special case of the fact that given a Kirby diagram for a $2$-handlebody $Y$, one obtains a Kirby diagram for $\overline{Y}$ by switching all crossings (that is, take a mirror of the link) and multiplying each framing coefficient by $-1$.
If $M^3(r)$ is $S^3$, then all three diffeomorphisms also hold with $X^4$ in place of $W^4$.
\end{remark}

\section{Triangle Groups and Artin Presentations}

We recall basic facts about triangle groups (for details, see Magnus~\cite[Ch.~II]{magnus}
and Ratcliffe~\cite[$\S$7.2]{ratcliffe}).
Let $l$, $m$, and $n$ be integers greater than or equal to $2$.
Let $\Delta=\Delta(l,m,n)$ be a triangle with angles $\pi/l$, $\pi/m$, and $\pi/n$.
Define
\[
	\delta = \frac{1}{l}+\frac{1}{m}+\frac{1}{n}
\]
The triangle $\Delta$ is: \emph{spherical} and lies in $X=S^2$ if $\delta>1$, \emph{Euclidean} and lies in $X=\R^2$ if $\delta=1$, and \emph{hyperbolic} and lies in $X=\H^2$ if $\delta<1$. The \emph{triangle reflection group} $T^{\ast}(l,m,n)$ is the group generated by the reflections of $X$ in the lines containing the sides of $\Delta$. The \emph{triangle group} $T(l,m,n)$ (sometimes called a \emph{von Dyck group}) is the index $2$ subgroup of $T^{\ast}(l,m,n)$ consisting of orientation preserving isometries of $X$. Geometrically, $T(l,m,n)$ is generated by the rotations of $X$ about the vertices of $\Delta$ by $2\pi/l$, $2\pi/m$, and $2\pi/n$ respectively. The triangle group $T(l,m,n)$ is presented by $\fg{x,y\mid x^l,y^m,(xy)^n}$. Notice that $T(l,m,n)$ is independent up to isomorphism of the order in which the integers $l$, $m$, and $n$ are listed.
The triangle group $T(l,m,n)$ is: \emph{spherical} and finite (but nontrivial) if $\delta>1$, \emph{Euclidean} and infinite if $\delta=1$, and \emph{hyperbolic} and infinite if $\delta<1$.
For example, $T(2,3,5)$ is the \emph{icosahedral group} isomorphic to the order $60$ alternating group $A_5$ on five letters. The infinite groups $T(3,3,3)$ and $T(3,3,4)$ correspond respectively to triangular tilings of the Euclidean and hyperbolic planes.

\begin{lemma}\label{tglemma}
Let $r=r(a,b,c)\in\mathcal{R}_2$. If $\card{a-c}$, $\card{b-c}$, and $\card{c}$ are all greater than or equal to $2$, then $\pi(r)$ is nontrivial.
If in addition $1/\card{a-c}+1/\card{b-c}+1/\card{c}\leq 1$, then $\pi(r)$ is infinite.
\end{lemma}

\begin{proof}
We construct a surjective group homomorphism
$\pi(r) \twoheadrightarrow T(\card{a-c},\card{b-c},\card{c})$.
Add the relation $(x_1x_2)^c$ to $r$ to obtain
\begin{align*}
\pi(r) &\twoheadrightarrow
\fg{x_1,x_2\mid x_1^{a-c}(x_1x_2)^c,x_2^{b-c}(x_1x_2)^c,(x_1x_2)^c}\\
&\cong \fg{x_1,x_2\mid x_1^{a-c},x_2^{b-c},(x_1x_2)^c}\\
&\cong \fg{x_1,x_2\mid x_1^{\card{a-c}},x_2^{\card{b-c}},(x_1x_2)^{\card{c}}}\\
&= T(\card{a-c},\card{b-c},\card{c})
\end{align*}
Now, apply properties of triangle groups recalled above.
\end{proof}

\begin{examples}
Consider the groups $\pi(r(-1,-3,2))$ and $\pi(r(10,1,3))$.
Both groups are perfect since their exponent sum matrices are unimodular.
Lemma~\ref{tglemma} implies that $\pi(r(-1,-3,2))$ is nontrivial and $\pi(r(10,1,3))$ is infinite.
The proof of Lemma~\ref{tglemma} shows that $\pi(r(-1,-3,2))$ surjects onto $T(3,5,2)\cong A_5$. 
\end{examples}

\begin{theorem}\label{tlist}
Let $r=r(a,b,c)\in\mathcal{R}_2$.
If $\pi(r)$ is trivial, then the $3$-tuple $(a,b,c)$ lies in the following list where $-(a,b,c)=(-a,-b,-c)$.
\begin{enumerate}
\item\label{t1} $(\pm1,\pm1,0)$ (four $3$-tuples)\\
\item\label{t2} $\pm(2,1,\pm1)$ and $\pm(1,2,\pm1)$ (eight $3$-tuples)\\
\item\label{t3} $\pm(1,5,2)$, $\pm(5,1,2)$, $\pm(2,5,3)$, $\pm(5,2,3)$ (eight $3$-tuples)\\
\item\label{t4} $(a,0,\pm1)$ and $(0,b,\pm1)$ where $a,b\in\Z$\\
\item\label{t5} $(c\pm1,c\mp1,c)$ where $c\in\Z$
\end{enumerate}
\end{theorem}

\begin{proof}
As $\pi(r)$ is trivial, $A(r)$ must be unimodular which means $ab-c^2=\pm1$.
Now, the basic idea is that either $\card{c}\leq1$ is small and $ab=c^2\pm1$ determines $a$ and $b$, or $\card{c}>1$ is larger and Lemma~\ref{tglemma} forces $a$ or $b$ to be close to $c$. 
We have $ab=c^2\pm1$ and, by Lemma~\ref{tglemma}, $\card{a-c}\le1$, $\card{b-c}\le1$, or $\card{c}\le1$. Notice that $(a,b,c)$ appears in the given list if and only if $-(a,b,c)$ appears.
Indeed, as $\pi(r(a,b,c))\cong\pi(r(-a,-b,-c))$ (see Examples~\ref{APex}\ref{charR2}),
our list must have this property.
So, it suffices to assume $c\geq0$ for the rest of the proof.
If $c=0$, then $ab=\pm1$, which gives the tuples~\ref{t1}.
If $c=1$, then $ab=0$ or $ab=2$. The former gives the tuples~\ref{t4}, and the latter gives the tuples~\ref{t2}.\\

Assume now that $c>1$. Then, $\card{a-c}\le1$ or $\card{b-c}\le1$,
and so $a$ or $b$ equals $c-1$, $c$, or $c+1$. If $a=c$, then $cb=c^2\pm1$ implies that $c|\pm1$, a contradiction. Similarly, $b\ne c$. Thus, $a=c\pm1$ or $b=c\pm1$.\newline
\textbf{Case 1: $ab=c^2+1$.} Suppose $a=c\pm1$.
Then, $ab=c^2+1$ implies $a|c^2+1$, and $a=c\pm1$ implies $a|c^2-1$.
So, $a|2$ and, as $a=c\pm1$ and $c>1$, we have $a=1$ or $a=2$. This gives the tuples $(1,5,2)$ and $(2,5,3)$. Similarly, $b=c\pm1$ gives the tuples $(5,1,2)$ and $(5,2,3)$.\newline
\textbf{Case 2: $ab=c^2-1$.} Then, $a=c\pm1$ if and only if $b=c\mp1$. This gives the tuples~\ref{t5}.
\end{proof}

\begin{remark}
It is not difficult to verify the converse of Theorem~\ref{tlist} directly using Tietze transformations.
This converse also follows from the Kirby calculus arguments in the next section.
Hence, Theorem~\ref{tlist} lists exactly the Artin presentations on two generators that present the trivial group.
\end{remark}

\section{4-manifolds}

In this section, we show that $M^3(r)$ is $S^3$ for each $r$ listed in Theorem~\ref{tlist}, and we identify the corresponding closed $4$-manifolds $X^4(r)=W^4(r)\cup_{\partial} D^4$.
First, we present a useful operation.

\begin{lemma}\label{slidelemma}
Let $r=r(a,b,c)\in\mathcal{R}_2$. There are diffeomorphisms
\begin{align}
W^4(r(a,b,c)) &\approx W^4(r(a+b-2c,b,b-c)) \label{firstdiffeo} \tag{$\dag$}\\
W^4(r(a,b,c)) &\approx W^4(r(a,a+b-2c,a-c)) \label{seconddiffeo} \tag{$\ddag$}
\end{align}
In particular, the corresponding $3$-manifolds $M^3(r)$ are diffeomorphic, and the corresponding groups $\pi(r)$ are isomorphic.
\end{lemma}

\begin{proof}
For the first diffeomorphism, proceed as shown in Figure~\ref{fig:slide}.  
\begin{figure}[htbp!]
    \centerline{\includegraphics[scale=1.0]{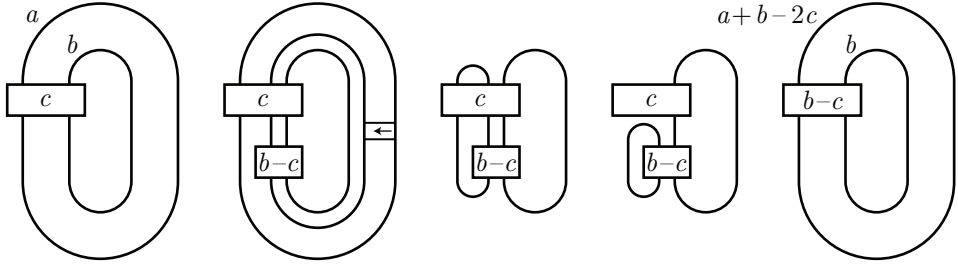}}
    \caption{From the left: Kirby diagram for $W^4(r(a,b,c))$, a $2$-handle slide, and results of two isotopies.}
\label{fig:slide}
\end{figure}
In the second diagram in Figure~\ref{fig:slide}, the middle circle is parallel to the $b$-framed circle and has linking number $b$ with it.
If the $a$- and $b$-framed circles are oriented clockwise, then the indicated $2$-handle slide is a handle subtraction; the framing of the $a$-framed circle changes to $a+b-2c$ (see~\cite[p.~141]{gs}).
\begin{figure}[htbp!]
    \centerline{\includegraphics[scale=1.0]{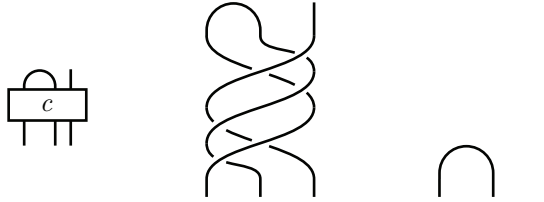}}
    \caption{From the left: a portion of the middle diagram in Figure~\ref{fig:slide}, the same portion (enlarged) with $c=1$ and without box notation, and the result of the isotopy of the portion.}
\label{fig:slide2}
\end{figure}
The isotopy of the middle diagram in Figure~\ref{fig:slide} is explained in Figure~\ref{fig:slide2}.
The result of Figure~\ref{fig:slide} is a Kirby diagram for \hbox{$W^4(r(a+b-2c,b,b-c))$}.
For~\eqref{seconddiffeo}, instead slide the $b$-framed circle over the $a$-framed circle in a similar manner.
The remaining claims in the lemma follow from~\eqref{firstdiffeo} and~\eqref{seconddiffeo} by taking boundaries.
\end{proof}

\begin{theorem}\label{fms}
For each Artin presentation $r$ listed in Theorem~\ref{tlist}, $M^3(r)$ is $S^3$.
Furthermore, the corresponding closed $4$-manifolds $X^4(r)$ are as follows.
\begin{enumerate}
\item\label{m1} $X^4(r(1,1,0))\approx\cp\cs\cp$ and $X^4(r(1,-1,0))\approx\cp\cs\cpb$\\
\item\label{m2} $X^4(r(2,1,1))\approx\cp\cs\cp$\\
\item\label{m3} $X^4(r(5,1,2))\approx X^4(r(5,2,3))\approx\cp\cs\cp$\\
\item\label{m4} $X^4(r(a,0,1))\approx
		\begin{cases}
			S^2\times S^2 	&\tn{if $a$ is even}\\
			\cp\cs\cpb 		 	&\tn{if $a$ is odd}
		\end{cases}$\\
		
\item\label{m5} $X^4(r(c+1,c-1,c))\approx
		\begin{cases}
			S^2\times S^2 	&\tn{if $c$ is odd}\\
			\cp\cs\cpb 		 	&\tn{if $c$ is even}
		\end{cases}$
\end{enumerate}
The $4$-manifolds for the remaining $3$-tuples in Theorem~\ref{tlist} are determined immediately from those just listed
and Remark~\ref{basicdiffeos}.
\end{theorem}

\begin{proof}
First,~\ref{m1} is clear since the Kirby diagrams are two-component unlinks with framings $\pm1$.
By~\eqref{firstdiffeo}, $W^4(r(2,1,1))\approx W^4(r(1,1,0))$, and~\ref{m2} now follows from~\ref{m1}.
Next, $W^4(r(2,1,-1))\approx W^4(r(2,1,1))$ by Remark~\ref{basicdiffeos},
$W^4(r(2,1,-1))\approx W^4(r(5,1,2))$ by~\eqref{firstdiffeo},
and $W^4(r(5,1,2))\approx W^4(r(5,2,3))$ by~\eqref{seconddiffeo}.
So,~\ref{m3} now follows from~\ref{m2}.
The Kirby diagram for $W^4(r(a,0,1)$ is a Hopf link with framings $a$ and $0$;
by~\cite[pp.~127, 130, \& 144]{gs}, the corresponding $4$-manifold may be closed up with a $4$-handle yielding
$S^2\times S^2$ if $a$ is even and $\cp\cs\cpb$ if $a$ is odd.
This proves~\ref{m4}.
Lastly, \eqref{seconddiffeo} gives $W^4(r(c+1,c-1,c))\approx W^4(r(c+1,0,1))$,
and~\ref{m5} now follows from~\ref{m4}.
\end{proof}

\begin{corollary}\label{cor_main}
The closed $4$-manifolds appearing as $X^4(r)$ for an Artin presentation $r$ on $n$-generators
for $n=0$, $1$, and $2$ are exactly:
$S^4$ for $n=0$, $\cp$ and $\cpb$ for $n=1$, and
$\cp\cs\cp$, $\cp\cs\cpb$, $\cpb\cs\cpb$, and $S^2\times S^2$ for $n=2$.
\end{corollary}

\begin{remark}
As noted by a referee, an alternative proof of Corollary~\ref{cor_main} may be obtained using
Corollary~1.4 from Meier and Zupan~\cite{mz}.
Their approach utilizes trisections of 4-manifolds and
does not appear to provide alternative proofs of Theorems~\ref{tlist} and~\ref{fms} herein.
\end{remark}

\end{document}